\documentclass[11pt,a4paper]{amsart}
\usepackage[T1]{fontenc}
\usepackage[latin9]{inputenc}
\usepackage{amssymb}
\usepackage{color}
\usepackage{graphicx,color}
\usepackage{amsmath, amssymb, graphics}

\makeatletter
\numberwithin{equation}{section} 
\numberwithin{figure}{section} 
  \@ifundefined{theoremstyle}{\usepackage{amsthm}}{}
  \theoremstyle{plain}
  \newtheorem{thm}{Theorem}[section]
  \theoremstyle{plain}
  
  \theoremstyle{plain}
  
  \theoremstyle{remark}
  \newtheorem{rem}[thm]{Remark}
  \theoremstyle{remark}

  \theoremstyle{plain}
  \newtheorem{lem}[thm]{Lemma}
  \theoremstyle{definition}

\providecommand{\keywords}[1]
{
\small	
\textbf{\textit{Keywords. }} #1
}
\newcommand{\R}{\mathbb{R}}

\begin{document}

\title[A family of Solutions of the n-Body problem]{Bifurcation of periodic orbits for the $N$-body problem, from a non geometrical family of solutions.}
\author[Oscar Perdomo, Andr\'es Rivera and Johann Su\'arez]
{Oscar Perdomo$^1$, Andr\'es Rivera$^2$, Johann Su\'arez$^3$}
\date{\today}

\address{$^1$ Department of Mathematics. 
	Central Connecticut State University. 
	New Britain. CT 06050}
\email{perdomoosm@cssu.edu}

\address{$^2$ Departamento de Ciencias Naturales y Matem\'aticas. 
	Pontificia Universidad \\
	Javeriana Cali, Facultad de Ingenier\'ia. 
	Calle 18 No. 118--250 Cali, Colombia}
\email{amrivera@javerianacali.edu.co}

\address{$^3$ Departamento de Matem\'aticas
	Universidad del Valle, Facultad de Ciencias. Cali, Colombia}
\email{johann.j.suarez@correounivalle.edu.co}

%
%

\begin{abstract}

Given two positive real numbers $M$ and $m$ and an integer $n>2$, it is well known that we can find a family of solutions of the $(n+1)$-body problem where the body with mass $M$ stays put at the origin and the other $n$ bodies, all with the same mass $m$, move on the $x$-$y$ plane following ellipses with eccentri\-city $e$. It is expected that this geometrical  family that depends on $e$, has some bifurcations that produces solutions where the body in the center moves on the $z$-axis instead of staying put in the origin. By doing analytic continuation of a periodic numerical solution of the $4$-body problem --the one displayed on the video  http://youtu.be/2Wpv6vpOxXk --we surprisingly discovered that the origin of this periodic solution is not part of the geometrical family of elliptical solutions parametrized by the eccentricity $e$. It comes from a not so geometrical but easier to describe family.  Having notice this new family, the authors find an exact formula for the bifurcation point in this new family and use it to show the existence of  non planar periodic solution for any pair of masses $M$, $m$ and any integer $n$. As a particular example we find a solution where three bodies with mass $3$ move around a body with mass $7$ that moves up and down.
\end{abstract} 

\maketitle

\section{Introduction}

The spatial isosceles solutions of the three body problem are solutions where two bo\-dies of equal mass $m$ have initial positions and velocities symmetric with respect to the $z$-axis which passes through their center of mass at the origin, the third body of mass $M$ moves on the $z$-axis, and at any time the three bodies form an isosceles triangle. The youtube video https://www.youtube.com/watch?v=fSmQyeKcj5k shows some images of a periodic spatial solution. Due to the importance and seve\-ral appli\-cations that the spatial three body configuration offers, nowadays there is a large numbers of papers devoted to this problem, see for example  \cite{Corbera2,D,OR,P1,P2,Rivera,X,YO} and specially \cite{MW} and the references therein. In all cited papers, the existence of periodic spatial isosceles solutions has been considered using variational methods, nume\-rical methods or techniques of analytical continuation. In most of the cases it is assumed that the body on the $z$-axis has a small mass $M$  or a mass $M=0$, see \cite{Corbera2,Galan,OR,Rivera}. In the latter case we obtain the \textit{Sitnikov problem}. For example in \cite{Corbera2} using the analytical continuation method of Poincaré, the authors prove analitically the existence of periodic and quasi-periodic spatial isosceles solutions,  for $M>0$ sufficiently small as a continuation of some well known periodic solutions of the \textit{reduced circular Sitnikov problem} (when the body in the $z$-axis has infinitesimal mass $M=0$ and the other bodies move in circular orbits of the two-body problem). Using variational methods and for $M>0$ non necessarily small, the authors in \cite{YO} prove numerically the existence of spatial isosceles solutions as minimizers of the corresponding Lagragian action functional.  A natural generalization of  the spatial isosceles solutions are those where $n$ bodies move on the vertex of a regular polygon perpendicular to the $z$-axis and the $(n+1)$-th body moves along the $z$-axis. An example of this type of solution can be seen in the youtube video https://www.youtube.com/watch?v=PtEMb6Rvflg.

In order to describe these motions more precisely, let us define
\[
\mathcal{R}[2\pi/n]=\begin{pmatrix}
\cos(2\pi/n)& -\sin(2\pi/n) & 0\\
\cos(2\pi/n)& \sin(2\pi/n) & 0\\
0 & 0 & 1\\
\end{pmatrix}.
\]
A direct computation (also see \cite{P}), shows that 

\begin{thm}\label{t7}
The functions
\[
x(t)=\left(0,0,f(t)\right),\quad y^{k}(t)=\mathcal{R}^{k}[2(k-1)\pi/n]y^{1}(t), \quad k=1,\dots, n,
\]
with 
\[
y^1(t)=\left(r(t)\cos(\theta(t))\, ,\, r(t)\sin(\theta(t)),\,-\frac{M}{m n}\, f(t)\right),
\]
satisfying 
$$f(0)=0,\quad \dot{f}(0)=b,\quad r(0)=r_0,\quad \dot{r}(0)=0,\quad \theta(0)=0 ,\quad \dot{\theta}(0)=\frac{a}{r^2_{0}},$$
provides a solution of the $(n+1)$-body problem with $n \geq 2$ and masses $M\geq 0$ for the body moving along $x(t)$ and mass $m>0$ for each body moving along the function $y^k(t)$, if and only if

\begin{equation}\label{e2}
\begin{split}
\ddot{f}&=-\frac{(M+mn)f}{h^3},\\
 \ddot{r} &= \frac{r_0^2 a^2 }{r^3}-\frac{m \lambda_n }{r^2}-\frac{Mr}{h^3},\\ \dot{\theta}&=\frac{r_0 a}{r^2},
\end{split}
\end{equation}


where  $\displaystyle{h=\sqrt{r^2+\Big(\frac{M+n m}{m n}\Big)^2f^2}}$ and $\displaystyle{\lambda_n=\sum_{k=1}^{n-1} \frac{1-\hbox{e}^{i \frac{2 \pi k}{n}}}{|\hbox{e}^{i \frac{2 \pi k}{n}}-1|^3}=\frac{1}{4}\sum_{k=1}^{n-1}\csc\Big(\frac{2 \pi k}{n}\Big).}$
\end{thm}

Related to the system (\ref{e2}) we mention some important results obtained in \cite{P,P1,P2,P3} by the first author of this paper. Firstly, a reduction of the system (\ref{e2}) to a single second order differential equation, see \cite{P}. Secondly, a rigorous mathematical proof of the periodicity of a solution of the 3-body problem describe by (\ref{e2}), see \cite{P1}. This was achieved through a numerical method that keeps track of the round-off error, developed in \cite{P3} and also by a lemma that can be viewed as a numerical version of the implicit function theorem, see \cite{P1}. Finally, in \cite{P2} the author describes a new family of symmetric periodic solutions of (\ref{e2}) that contains a nontrivial bifurcation point. The diagram of periodic solutions in the space of initial conditions (similar to the one in Figure \ref{fig1}) shows  four branches emanating from this bifurcation solution. One is unbounded, another has as a  limit point a solution where there is collision of two bodies, the third branch has as a limit point a solution where there is collision of the three bodies and the fourth one ends on a family of trivial solutions.

The main purpose of this paper is to provide an analysis of the properties of the family of periodic solutions in the fourth branch mentioned above and to provide an explicit formula for its limit point. To this end, assuming nonzero angular momentum for (\ref{e2}) we will consider an appropriate system which symmetries can be used to obtain the periodic solutions in the fourth branch; see Section 2. In Section 3, we prove Theorems \ref{red1} and \ref{red2} which are the statements of our main results. It is interesting to compare our results in this paper with those in \cite{Corbera2}. Both papers analytically show the existence of periodic solutions using analytic continuation techniques. The new solutions in \cite{Corbera2} emanate from solutions of the Sitnikov problem while those in this paper emanates from what we can call ``\textit{part of circular solutions}'' where the two masses $m$ and $M$, are arbitrary. We call them part of circular solutions because $n$ of the $(n+1)$ bodies move on a circle but they do not necessarily cover the whole circle. Finally, numerical validation of the theoretical results  obtained in Section 3 are presented in Section 4.

\section{The reduced problem and symmetries}
A simple inspection of the system (\ref{e2}), shows that $r^2\dot{\theta}$ is the first integral of the angular momentum $C=r_{0}a$. Having said this, throughout this document we consider only solutions of (\ref{e2}) with nonzero angular momentum, i.e., $C=r_{0}a\neq 0$. This assumption allows us to consider only the initial value problem,
\begin{equation}
\begin{split}\label{reduced1}
\ddot{f}&=-\frac{(M+mn)f}{h^3}, \qquad \qquad \, f(0)=0,\quad \dot{f}(0)=b,\\
\ddot{r} &= \frac{r_0^2 a^2 }{r^3}-\frac{m \lambda_n }{r^2}-\frac{Mr}{h^3}, \quad \quad r(0)=r_0,\quad \dot{r}(0)=0.
\end{split}
\end{equation}

We can check that, if $\phi(t)=(f(t), r(t))$ is a solution of (\ref{reduced1}) and 
\[
\theta(t)=\int_{0}^{t}\frac{r_{0}a}{r^{2}(s)}ds,
\]
then $\overline{\phi}(t)=(f(t), r(t),\theta(t))$ is a solution of (\ref{e2}). From now on, we denote by 
\[
F(a,b,t)=f(t), \quad R(a,b,t)=r(t), \quad \Theta(a,b,t)=\theta(t),
\] 

the solutions of the system (\ref{e2}) with initial conditions
\begin{equation}\label{initial conditions}
f(0)=0,\quad \dot{f}(0)=b,\quad r(0)=r_0,\quad \dot{r}(0)=0,\quad \theta(0)=0 ,\quad \dot{\theta}(0)=\frac{a}{r^2_{0}}.
\end{equation}

\begin{rem}\label{rem1}
	If for some $a$ and $b$, $f(t)$ and $r(t)$ are $T$-periodic of the system (\ref{reduced1}) then $(f(t),r(t),\theta(t))$ defines a periodic solution of the $(n+1)$-th body, if and only if $\theta(T)$ is equal to $n_{1} \pi/n_2$ with $n_1$ and $n_2$ whole numbers. See \cite{P2}. In general, $T$-periodic solutions of the systems (\ref{reduced1}) define \textit{reduced-periodic} solutions of the $(n+1)$-body problem. This is, solutions with the property that every $T$ unites of time, the positions and velocities of the $(n+1)$ bodies only differ by an rigid motion in $\R^{3}.$
\end{rem}
The existence of periodic solutions of (\ref{e2}) becomes simpler if we restrict our seach to periodic solutions with symmetries. In such a case,  the following lemma, see \cite{P2} provides a useful result.

\begin{lem}\label{simmetry} Let $\phi(t)=(F(a,b,t), R(a,b,t))$ be a solution of (\ref{reduced1}).
\begin{enumerate}
\item[$\rhd$]If for some $0<T$ we have

		\begin{equation*}\label{s1}
		F(a,b,T)=0 \quad \text{and} \quad R_{t}(a,b,T)=0, \qquad (\dagger) 
		\end{equation*}\\
then $f(t)=F(a,b,t)$ and $r(t)=R(a,b,t)$ are both $2T$-periodic functions.\\
\item[$\rhd$] If for some $0<T$ we have

		\begin{equation*}\label{s2}
		F_{t}(a,b,T)=0 \quad \text{and} \quad R_{t}(a,b,T)=0, \qquad (\dagger \dagger) 
		\end{equation*}\\
then $f(t)=F(a,b,t)$ and $r(t)=R(a,b,t)$ are both $4T$-periodic functions.
\end{enumerate}
\end{lem}

It is worth to mentioning that the solutions $\phi(t)=(F(a,b,t), R(a,b,t))$ that satisfies $(\dagger)$ are called \textit{odd} solutions because  $f(t)=F(a,b,t)$ is an odd function with respect to $t=0$. On the other hand, if $\phi(t)$ satisfies $(\dagger\dagger)$ they are called \textit{odd/even} solutions because $f(t)=F(a,b,t)$ is an odd function with respect to $t=0$, but with respect to $t=T$,  both functions  $f(t)=F(a,b,t)$ and $r(t)=R(a,b,t)$ are even. Furthermore, we point out that every odd/even solution is also an odd solution.

\section{Main results}

\subsection{Periodic solutions for the reduced problem} In this section we prove the existence of a one parametric a family of periodic solutions for the reduced $(n+1)$-body problem (\ref{reduced1}). 
\begin{thm}\label{red1} For any $n\ge 2$, let us define $\displaystyle{\lambda_n=\frac{1}{4}\sum_{k=1}^{n-1}\csc\big(k\pi/n\big)}$. Assume that $m, M$ and $n$ satisfy that $\lambda_n\ne n p^2+\frac{M}{m}(p^2-1)$ for every positive integer  $p$. Then, for any positive real number $r_0$ there exist  $b\ne0$ near $0$,   $T>0$ near $\pi \sqrt{\frac{r_0^3}{n m+M}}$, and $a>0$ near $\sqrt{\frac{\lambda_n m+M}{r_0}}$ that provides an odd $2T$-periodic solution of the reduced $(n+1)$-body problem with initial conditions described in Theorem \ref{t7}.
\end{thm}

\begin{proof} 
	
	 For fixed values of $m,M,r_0$, let $F(a,b,T)$, $R(a,b,T)$ be the solutions of (\ref{reduced1})-(\ref{initial conditions}) evaluated at $t=T$. By Lemma \ref{simmetry} it follows that if for some $a,b$ and $T$ we have that
	\begin{eqnarray*}\label{eq zeros}
	F(a,b,T)=0\quad\hbox{and}\quad R_t(a,b,T)=0,
	\end{eqnarray*}
	
	then $t\longrightarrow f(t)=F(a,b,t)$ and $t\longrightarrow r(t)=R(a,b,t)$ are odd and even functions respectively and both functions are periodic with period $2T$. An easy computation shows that
	\[
	F(a,0,T)=0, \quad \forall T\in \R.
	\]
	
	From here we can deduce the following
	\begin{itemize}
		\item[a)]  For all $(a,b,T)$ we can express $F$ as 
		\begin{eqnarray}\label{def of tilde F new}
		F(a,b,T)=b \tilde{F}(a,b,T), 
		\end{eqnarray}
		with $\tilde{F}$ some smooth function. Therefore, an easy a direct computation shows that $\displaystyle{\tilde{F}(a,0,T)=F_{b}(a,0,T)}$. Moreover, for all $(a,T)$ it follows
		
		\begin{equation*}
	     \tilde{F}_{t}(a,0,T)=F_{bt}(a,0,T), \quad 
		\tilde{F}_{a}(a,0,T)=F_{ba}(a,0,T), \quad \text{and} \quad  \tilde{F}_{b}(a,0,T)=F_{bb}(a,0,T)/2.
		\end{equation*}	
		\,
	    \item[b)] If $a=\displaystyle{a_{0}=\sqrt{\frac{\lambda_{n}m+M}{r_0}}}$, we have that $F(a_{0},0,T)=0$ and $R(a_{0},0,T)=r_0$ solves (\ref{reduced1}). Therefore,
	    
		\[
		F(\alpha(T))=R_{t}(\alpha(T))=0, \quad \text{with} \quad \alpha(T)=(a_{0},0,T), \quad T\in \R.
		\]\,
		
		\noindent
		The path $\alpha(T)$, more precisely, the pseudo periodic solutions of the $(n+1)$-body problem induced by the equations $(\dagger)$,  can be viewed geometrically as the pseudo periodic solutions where $n$ of the $(n+1)$-bodies move along part of a circle and the $(n+1)$-body stays put 
		in the center. From this collection we will find a bifurcation point, a particular value for $T$, that will provide the nontrivial periodic solutions.
	\end{itemize} 

    The previous observations suggest to study the solutions of the system
    \begin{eqnarray}\label{eq zeros2}
    \tilde{F}(a,b,T)=0 \quad \hbox{and}\quad R_{t}(a,b,T)=0.
    \end{eqnarray}
    
    To this end we will use relation (\ref{def of tilde F new}) to compute derivatives of $\tilde{F}$ and $R_{t}$. Recall that 
	\begin{eqnarray}\label{eFnew}
	F_{tt}&=&-\frac{(M+mn)F}{h^{3}},\\ \label{eRnew}
	R_{tt}&=&\frac{a^2 r_0^2}{R^3}-\frac{m \lambda_n }{R^2}-\frac{MR}{h^3},
	\end{eqnarray}\\
	where $\displaystyle{h=\sqrt{R^2+\Big(\frac{M+n m}{m n}\Big)^2F^2}}$. Taking the partial derivative with respect to $b$ on both sides of Equation (\ref{eFnew}) and evaluating at $\alpha(t)$ give us that $F_b(\alpha(t))$ satisfies
	\[
	\ddot{u}=-\frac{(M+mn)}{r_{0}^{3}}u, \quad \text{with} \quad u(0)=0, \,\, \dot{u}(0)=1.
	\]
	
	Therefore, 
	\begin{eqnarray}\label{dfbnew}
	F_b(\alpha(t))=\left(\frac{M+m n}{r_{0}^{3}}\right)^{-1/2} \sin \left(\left(\frac{M+m n}{r_{0}^{3}}\right)^{1/2} t\right).
	\end{eqnarray}
	The equation above shows that, if 
	\begin{eqnarray}
	T_{0}=\pi\sqrt{\frac{ r_0^3}{m n+M}}\, ,
	\end{eqnarray}
	
	then 
	\[
	\tilde{F}(\alpha(T_{0}))=F_{b}(\alpha(T_{0}))=0.
	\]
	
	In consequence, from $a)$ and $b)$ it follows that $\alpha(T_{0})$ satisfies
   \[
   \tilde{F}(\alpha(T_{0}))=0,  \quad \hbox{and}\quad R_{t}(\alpha(T_{0}))=0,
   \]
   
   showing that the point $(a_{0},0,T_{0})$ solves (\ref{eq zeros2}). Finally, from Equation (\ref{dfbnew}) we get that
	\begin{eqnarray}\label{dFtildetnew}
	\tilde{F}_t(\alpha(T_{0}))=F_{bt}(\alpha(T_{0}))=-1.
	\end{eqnarray}

	\vspace{0.2 cm}
	\noindent
	Now, we take the derivative with respect to $a$ on both sides of Equation (\ref{eRnew}) and evaluate at  $\alpha(t)$. Then $R_a(\alpha(t))$ satisfies
	\[
	\ddot{v}=2\left(\frac{\lambda_{n}m + M}{r_{0}^{3}}\right)^{1/2}-\frac{(\lambda_{n}m+M)}{r_{0}^{3}}v, \quad \text{with} \quad v(0)=0, \,\, \dot{v}(0)=0,
	\]
	therefore
	\[
	R_a(\alpha(t))=2\left(\frac{\lambda_{n}m+M}{r_{0}^{3}}\right)^{-1/2}\left(1- \cos \left(\left(\frac{\lambda_{n}m+M}{r_{0}^{3}}\right)^{1/2} t\right)\right)
	\]
	From this last expression we deduce
	\[
	R_{at}(\alpha (t))=2 \sin \left(\left(\frac{\lambda_{n}m+M}{r_{0}^{3}}\right)^{1/2} t\right).
	\]
	Applying the same ideas, the function $\displaystyle{R_b(\alpha(t))}$ can be obtained by solving
	\[
	\ddot{w}(t)=-\frac{(\lambda_{n}m+M)}{r_{0}^{3}}w, \quad \text{with} \quad w(0)=0, \,\, \dot{w}(0)=0,
	\]
	therefore
	\[
	R_b(\alpha(t))=R_{bt}(\alpha(t))=0,
	\] 
	for all $t\in \R.$ Further, from (\ref{eRnew}) we can deduce that  $\displaystyle{R_{tt}(\alpha(t))=0}$ for all $t\in \R.$
These computations shows that the gradient of the function $R_{t}$ at $\alpha(T_{0})$ is given by

\begin{eqnarray}\label{gRt}
\nabla R_{t}(\alpha(T_{0}))=\left(2 \sin\left(\pi\sqrt{\frac{\lambda_{n}m+M}{mn+M}}\right),0,0\right).
\end{eqnarray}

Notice that our condition of $\lambda_n$ guarantees that $\nabla R_t(\alpha(T_{0}))$ does not vanish. On the other hand, using $a)$ and the equation (\ref{dFtildetnew}) we obtain
\begin{eqnarray}\label{gRtnew}
\nabla \tilde{F}(\alpha(T_{0}))=\left(F_{ba}(\alpha(T_{0})),\frac{1}{2}F_{bb}(\alpha(T_{0})),-1\right).
\end{eqnarray}

Since 
$$\eta=\nabla \tilde{F}(\alpha(T_{0}))\times \nabla R_t(\alpha(T_{0}))=\left( 0,R_{ta}(\alpha(T_{0})),\frac{1}{2} R_{ta}(\alpha(T_{0})) F_{bb}(\alpha(T_{0})) \right),$$

has its second entry different from zero, by the Implicit Function Theorem there exists  $\epsilon >0$ and a pair of continuous functions such that $T,a:(-\epsilon,\epsilon) \to \R$ 
\begin{equation*}
b\longrightarrow  T=T(b), \quad  b\longrightarrow  a=a(b),
\end{equation*}

with $T(0)=T_{0}$, $a(0)=a_{0}$, such that
\[
\tilde{F}(\gamma_1(b))=0, \quad \text{and} \quad R_t(\gamma_1(b))=0,
\]
with $\gamma_1:(-\epsilon,\epsilon) \to \R^3,$ given by
\begin{equation}\label{gamma1 curve}
b\to \gamma_1(b)=(a(b),b,T(b)), \quad \quad \gamma_1(0)=(a_{0},0,T_{0}).
\end{equation}

Therefore, for each $b\in(-\epsilon,\epsilon)$ it follows
\begin{equation*}
\begin{split}
F(\gamma_1(b))=b\tilde{F}(\gamma_1(b))&=0,  \quad \text{and} \quad R_{t}(\gamma_1(b))=0.
\end{split}
\end{equation*}
In consequence, using Lemma \ref{simmetry} we get that for any $b$, the functions $f(t)=F(a,b,t)$ y $r(t)=R(a,b,t)$ define a  $2T(b)$-periodic solution of the reduced problem (\ref{reduced1}).

%
%
%

\end{proof}

\begin{thm}\label{red2} For any $n\ge 2$, let us define $\displaystyle{\lambda_n=\frac{1}{4}\sum_{k=1}^{n-1}\csc\big(k\pi/n\big)}$. Assume that $m,M$ and $n$ satisfy that $\lambda_n\ne n q^2+\frac{M}{m}(q^2-1)$ for every positive even integer $q$. Then, for any positive real number $r_0$ there exist   $b\ne0$ near $0$,   $T_{*}>0$ near $\dfrac{\pi}{2}\sqrt{\frac{r_{0}^{3}}{n m+M}}$ and $a>0$ near $\sqrt{\frac{\lambda_n m+M}{r_0}}$ that provides an ood/even  $4T_{*}$-periodic solution of the reduced $(n+1)$-body problem with initial conditions described in Theorem \ref{t7}.
\end{thm}

\begin{proof}
	We follows the same lines of the proof of Theorem \ref{red1}, but now considering consider the system
	\begin{equation}\label{s2}
	F_t(a,b,T)=0 \quad \text{and} \quad R_{t}(a,b,T)=0.\\
	\end{equation}
		
	By (\ref{def of tilde F new}) there exists a function $W(t,a,b)$ such that
	\begin{equation*}
	\begin{split}
	F_t(a,b,t)&=bW(a,b,t),\\
	F_{bt}(a,b,t)&=W(a,b,t)+bW_{b}(a,b,t),
	\end{split} 
	\end{equation*}
	
	for all $(a,b,t). $ In particular
	\begin{equation*}
	F_{bt}(a,0,t)=W(a,0,t),
	\end{equation*}
	
 for all $(a,t)$. Now we search for points $(a,b,T_{*})$ such that 
	\[
	W(a,b,T_{*})=0 \quad \text{and} \quad R_{t}(a,b,T_{*})=0.
	\]
	
	To this end, we need to study the zeroes of the function $F_{bt}(a,b,T_{*})$. From (\ref{dfbnew}) we have
	\begin{equation*}\label{ecu de ut}
	F_{bt}(\alpha(t))=\cos\left(\left(\frac{M+m n}{r_{0}^{3}}\right)^{1/2} t\right),
	\end{equation*}
    Therefore, $F_{bt}(\alpha(T_{0}/2))=0$ and,
	\[
	W(a_{0},0,T_{0}/2)=0 \quad \text{and} \quad R_t(a_0,0,T_{0}/2)=0.
	\]
	
	Once again, with the aim of applying the Implicit Function Theorem we consider the gradiente vector $\nabla W(a_{0},0,T_{0}/2)$ and $\nabla R_{t}(a_{0},0,T_{0}/2)$. A direct computation shows that 
$$\nabla W(\alpha(T_{0}/2))\times \nabla R_t(\alpha(T_{0}/2))=\left( 0,-\frac{\pi}{T_{0}}R_{ta}(\alpha(T_{0}/2)),\frac{1}{2} R_{ta}(\alpha(T_{0}/2)) F_{bbt}(\alpha(T_{0}/2)) \right).$$

Notice that the second entry of above vector is given by
\[
R_{ta}(\alpha(T_{0}/2))=2 \sin\left(\frac{\pi}{2}\sqrt{\frac{\lambda_{n}m+M}{mn+M}}\right),
\]

which by hypothesis is different from zero. Then, there exist $\delta>0$ and two conti\-nuous functions $T_{*},a_{*}:(-\delta,\delta)\rightarrow\R$ such that
	\begin{equation*}
	b\rightarrow T_{*}=T_{*}(b), \quad  b\rightarrow a_{*}=a_{*}(b),
	\end{equation*}
	
 with $T_{*}(0)=T_{0}/2$, $a_{*}(0)=a_{0}$ and
	\[
	W(\gamma_2(b))=0 \quad \text{and} \quad R_{t}(\gamma_{2}(b))=0,
	\]
	with $\gamma_2:(-\delta,\delta) \to \R^3,$
	
	\begin{equation}\label{gamma curve2}
	b\to \gamma_2(b)=(a_{*}(b),b,T_{*}(b)), \quad \quad \gamma_2(0)=(a_{0},0,T_{0}/2).
	\end{equation}
	
Then for $b\in(-\delta,\delta)$ it follows
\begin{equation*}
\begin{split}
F_t(\gamma_{2}(b))=bW(\gamma_{2}(b))&=0,  \quad \text{and} \quad R_{t}(\gamma_{2}(b))=0,
\end{split}
\end{equation*}

Once again, by Lemma \ref{simmetry} we have that for any $b$, the functions $f(t)=F(a,b,t)$ y $r(t)=R(a,b,t)$ define an odd/even $4T_{*}(b)$-periodic solution of the reduced problem (\ref{reduced1}). 
\end{proof}


\subsection{Periodic solutions for the $(n+1)$-body problem}

We will only consider the case when the  solutions of the $(n+1)$-body problem comes from solutions of Equation $(\dagger)$. The case when the solutions of the $(n+1)$-body problem  come from solutions of equation  $(\dagger \dagger)$ is similar.

In order to show that we can find a non trivial periodic solution with $b>0$ we just need to check that the function $\Theta(a,b,t)$ is not constant along the curve $\gamma_{1}(b)=(a(b),b,T(b))$ defined in Equation (\ref{gamma1 curve}). This is true because every triple $(a(b),b,T(b))$ solves the equation $F=0$ and $R_t=0$ and by continuity, if $b\to \Theta(a(b),b,T(b)))$ is not constant, then we can find a $b$ such that $\gamma_{1}(b)$ satisfies that $\Theta(\gamma_{1}(b))=n_1\pi/n_2$ with $n_1$ and $n_2$ integers. See Remark \ref{rem1}.

One way to study the behavior of the function $\Theta$ along the curve $\gamma_{1}(b)$ is by reparametrizing the curve $\gamma_{1}(b)$ as 
$\beta(\tau)=\gamma_{1}(b(\tau))$ where $\beta$ is an integral curve of the vector field $X$ defined as

$$X=\nabla \tilde{F}\times \nabla R_t=\left( \tilde{F}_b R_{tt}-\tilde{F}_t R_{tb},\tilde{F}_t R_{ta}-\tilde{F}_a R_{tt},\tilde{F}_a R_{tb}-\tilde{F}_b R_{ta} \right).$$

As we have done before, we can explicitly  compute as many partial derivatives of the functions $\tilde{F}$ and $R_t$ at the bifurcation point $p_0=\beta(0)=\gamma(0)$. Therefore we can compute as many derivatives as needed for the curve $\beta$ at $\tau=0$ and since we can also compute as many partial derivatives of the function $\Theta$ at $p_0$ then using the chain rule we can compute the first and second derivative of the function 
$$\xi(\tau)=\Theta(\beta(\tau)),$$
at $\tau=0$. A direct computation shows that $\xi^\prime(0)=0$ and a long direct computation, see \cite{Suarez}, shows that
\begin{eqnarray*}
	\xi^{\prime \prime}(0)&=& \big(A(n,m,M,r_0)+B(n,m,M,r_0)R_{at}(\alpha(T_{0}))\big)R^{2}_{at}(\alpha(T_{0})),
\end{eqnarray*}
with $A=A(n,m,M,r_0)$ and $B=B(n,m,M,r_0)$ given by
\begin{equation*}
	\begin{split}
		A&=\frac{3r_{0}^{5/2}\pi\bigg[9(\lambda_{n}m+M)\Big(\lambda_{n}m+24M-4mn+16M\sqrt{\frac{\lambda_{n}m+M}{r_{0}^{3}}}\Big)-8M(M+mn)\Big(9+8\sqrt{\frac{\lambda_{n}m+M}{r_{0}^{3}}}\Big)\bigg]}{16m^{2}n^{2}(-\lambda_{n} m+4mn+3M)(\lambda_{n}m+M)^{1/2}}\\ \vspace{0.7 cm}
		B&=\frac{9Mr_{0}^{5/2}(M+mn)^{5/2}\Big(3+2\sqrt{\frac{\lambda_{n}m+M}{r_{0}^{3}}}\Big)}{m^{2}n^{2}(\lambda_{n}m+M)(\lambda_{n}m-3M-4mn)^{2}}.
	\end{split}
\end{equation*}


Which implies that for most of the choices of $n$, $M$ and $m$ the second derivative of $\xi(\tau)=\Theta(\beta(\tau))$ at $\tau=0$ is different from zero.

\section{Numerical solutions}

In this section we present the analytic continuation of a periodic solution that is close to the solution displayed in the youtube video {\color{blue}http://youtu.be/2Wpv6vpOxXk}. It can be checked that  $m=92$, $M=242$, $n=3$, $r_0=11$ and

$$q_0=(a_0,b_0,T_0)=(1.84153, 3.79392, 7.31715),$$

provides a periodic solution. Figure \ref{fig3} shows the trajectory of one of the three bodies with mass 92.

 \begin{figure}[hbtp]
\begin{center}\includegraphics[width=.45\textwidth]{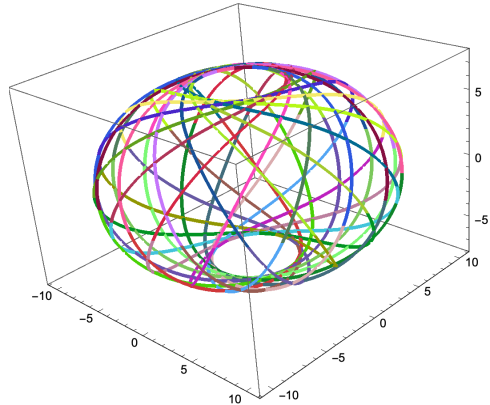}
\end{center}
\caption{Trajectory of one of the bodies with mass 92 for the solution of the 4 body problem with initial conditions $q_0=(a_0,b_0,T_0)=(1.84153, 3.79392, 7.31715)$   }\label{fig3}
\end{figure}

When we extend this solution, similar to the results in \cite{P2}, the collection of points $(a,b,T)$ that satisfy the equations $F(a,b,T)=0$ and $R_{t}(a,b,T)=0$ with $b\ne 0$ have the shape of a fork, where two of the edges goes to a points with $a=0$ corresponding to periodic solutions with collisions, the other edge seems to be unbounded and the fourth edge goes to a point where $b=0$. The value of $a$ for this limit point with $b=0$ is $a_{0}=\sqrt{\frac{\lambda_3 m+M}{r_0}}=\sqrt{22+\frac{92}{11 \sqrt{3}}}\approx 5.17965$ which was somehow expected because this is the $a$ corresponding to the solution where the mass in the center stays put and the other masses move along a circle. The value of $T$ for this limit point is near 5.03224. In other words the limit point with $b=0$ that we found by analytic continuation of the periodic solution displayed in the video is

$$q_1=(a_1,b_1,T_1)= (5.17965,\, 0,\, 5.03224).$$

Figure \ref{fig1} shows the points $q_0$ and $q_1$ as part of the family of reduced periodic solutions of the 4-body problem. This paper comes from an effort to understand the reason of the value of $T$ in the point above. Initially we were expecting a value for $T$ equal to $11 \sqrt{\frac{33}{92 \sqrt{3}+726}} \pi\approx 6.67179$ which is half of the period of the circular solution where the mass in the center stays still. The result of this explanation is Theorem \ref{red1}, it explain how actually this limit point is a bifurcation of the solutions coming from the line 
$$L=\Big\{(a_{0},0,t)=(\sqrt{22+\frac{92}{11 \sqrt{3}}},0,t):t\in\mathbb{R}\Big\},$$

 which even though they do not provide geometric periodic solutions, they satisfy the equation $F(a,b,t)=0$ and $R_{t}(a,b,t)=0$. Our main theorem predicts the value of bifurcation for $T$ at $11 \sqrt{\frac{11}{518}} \pi\approx 5.03586$ which completely explains our numerically value of $T$ in the point $q_1$. Figure \ref{fig1} shows the solutions with $b\ne 0$ that form a fork shape curve and the solution with $b=0$, the line $L$ defined above. It is worth pointing out that this line $L$ is the image of the curve $\alpha$ that we used in the proof of the main results of this paper.

 \begin{figure}[hbtp]
\begin{center}\includegraphics[width=.32\textwidth]{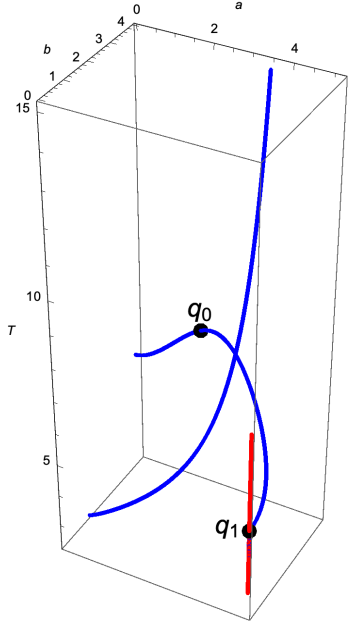}
\end{center}
\caption{Two bifurcation points for the spatial 4-body problem. For these solutions the masses of three of the bodies is 92 and the mass of the fourth body is 242.  One of the bifurcations is explained in the paper \cite{P2}, and the other bifurcation, the one that takes place on the line $L$, is explained in this paper.}\label{fig1}
\end{figure}

Just to give an application to our Theorem \ref{red1}, we decided to consider periodic solutions of the 4-body problem produced when $m=3$, $M=7$ and $r_0=11$. According to Theorem \ref{red1}, when we consider the system 
\begin{equation}\label{thesys}
\begin{cases}\begin{split}F(a,b,T)&=0,\\ 
R_{t}(a,b,T)&=0,
\end{split}
\end{cases}
\end{equation}

it follows that, along the line of solutions of this system 
\[
\Big\{(a_0,b,T):a_0=\sqrt{\frac{1}{11} \left(\sqrt{3}+7\right)},\, b=0\Big\},
\] 
there is a bifurcation at $$p_0= (a_0,0,T_0)=\Big(\sqrt{\frac{1}{11} \left(\sqrt{3}+7\right)},0,\frac{11 \sqrt{11} \pi }{4}\Big)\approx (0.890967,0,28.6536).$$

In order to find a non trivial solution near $p_0$, we took $b=0.05$ and we did a search by doing small changes of $a_0=0.890967$ and $T_0=28.6536$ to solve the system (\ref{thesys}). We found that the point
$$p_1=(a_1,b_1,T_1)=(0.8892815, 0.05, 28.708),$$

Satisfies that $|F(a_1,b_1,T_1)|$ and $|R_{t}(a_1,b_1,T_1)|$  are smaller than $10^{-6}$. Using the point $p_1$ we did an analytic continuation to obtain solutions of the system (\ref{thesys}) with $b\ne 0$. Figure \ref{bif}
shows about 12,000 solutions found numerically, along with the solutions along the line $b=0$.

 \begin{figure}[hbtp]
\begin{center}\includegraphics[width=.15\textwidth]{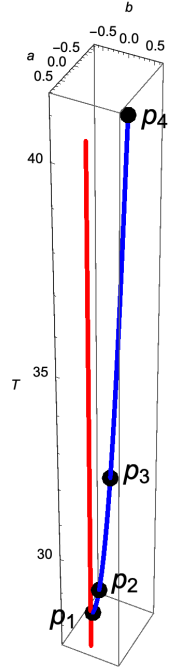}
\end{center}
\caption{Solutions of the system (\ref{thesys}) near the bifurcation point $p_0$. The point $p_1$ allowed us to start the analytic continuation by moving along the integral curve of the vector field $\nabla F\times \nabla R_{t}$. This figure shows about 12,000 points with $b\ne0$.  The points $p_1$, $p_2$ and $p_3$ are points where $\Theta$ equals to $\frac{3\pi}{4}$, $\frac{4\pi}{5}$, and $\pi$ respectively.   }\label{bif}
\end{figure}

We can compute the exact value for $\Theta(p_0)$, it is $\frac{ a_0 T_0}{11}=\frac{\pi \sqrt{7+\sqrt{3}}}{4}\approx 2.32086$. We also have that  $\Theta(p_1)\approx 2.32327$, this time the computation has to be done numerically. 

 \begin{figure}[hbtp]
\begin{center}\includegraphics[width=.6\textwidth]{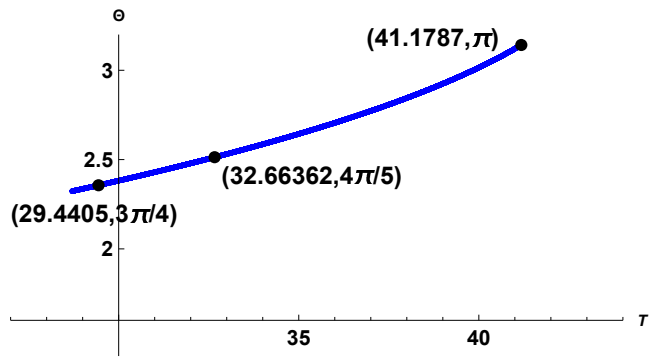}
\end{center}
\caption{This is the graph of  $\Theta$ against $T$ along the solution of the system  $F(a,b,T)=0,\,  R_{t}(a,b,T)=0$. We have highlighted three points that represents periodic solutions.  }\label{thvsT}
\end{figure}

Since the function $\Theta$ along the solutions of the system (\ref{thesys}) changes, by continuity we have that for some solution $p=(a,b,T)$ of the system $\Theta(p)= \frac{u}{v} 2 \pi$ with $u$ and $v$ whole numbers. Therefore we get that the functions

\begin{eqnarray*}
x(t) &=&\left(0,0,F(a,b,t) \right)\\
y^1(t)&=& \left(R(a,b,t)\cos(\Theta(a,b,t))\, ,\, R(t)\sin(\Theta(a,b,t)),\,-\frac{11}{9}\, F(a,b,t)\right) \\
y^2(t)&=& \left(R(a,b,t)\cos(\Theta(a,b,t)+\frac{2\pi}{3})\, ,\, R(t)\sin(\Theta(a,b,t)+\frac{2\pi}{3}),\,-\frac{11}{9}\, F(a,b,t)\right) \\ 
y^3(t)&=& \left(R(a,b,t)\cos(\Theta(a,b,t)+\frac{4\pi}{3})\, ,\, R(t)\sin(\Theta(a,b,t)+\frac{4\pi}{3}),\,-\frac{11}{9}\, F(a,b,t)\right), \\
\end{eqnarray*}

will eventually close and therefore the solution of the $4$-body problem given by these functions will be, not only pseudo periodic, but periodic. We have noticed that along the branch of solutions of the system (\ref{thesys}) with $b\ne0$ (see Figure \ref{bif}),  $\Theta$ is a function of $T$, see Figure \ref{thvsT}. After noticing that the values of $\Theta$ increase up to numbers bigger that $\pi$, we searched for the solutions that satisfy $\theta(T) =3\pi/4$, $\theta(T) =4\pi/5$ and $\theta(T) =\pi$. We found that the points 
\begin{eqnarray*}
p_2&=&(a,b,T)=(0.866953, 0.187583, 29.4405), \\
p_3&=&(a,b,T)=(0.775642, 0.400635, 32.6636), \\
p_4&=&(a,b,T)=(0.547954, 0.634946, 41.1787), 
\end{eqnarray*}

satisfy that $\Theta(p_1)=\frac{3\pi}{4}$, $\Theta(p_3)=\frac{4\pi}{5}$, and   $\Theta(p_4)=\pi$. Figure \ref{threeorbits} shows the image of the function $y^1(t)$ for these three solutions $p_1$, $p_2$ and $p_3$.

 \begin{figure}[hbtp]
\begin{center}\includegraphics[width=.3\textwidth]{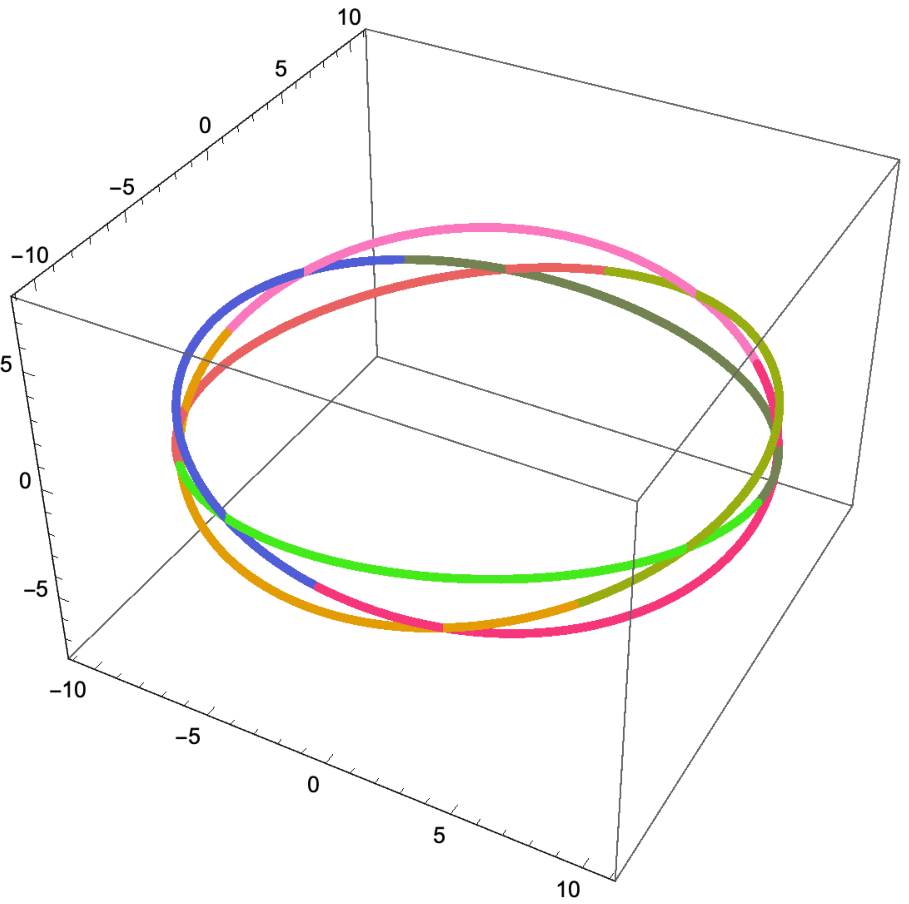} \includegraphics[width=.3\textwidth]{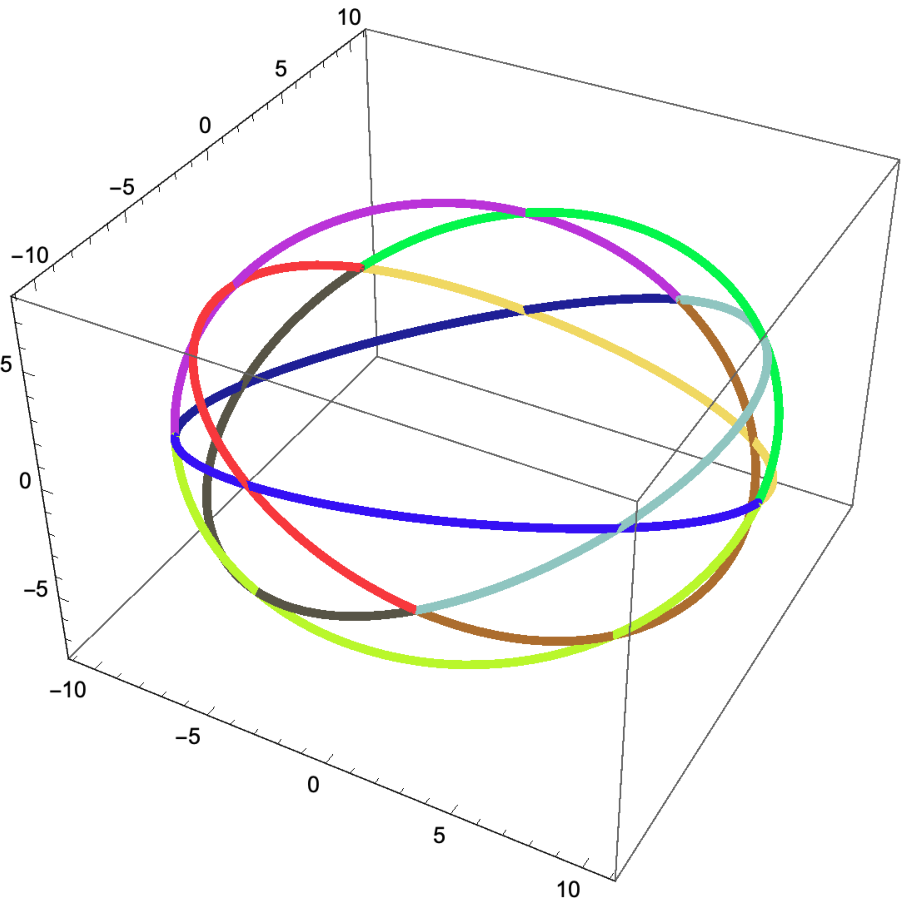} \includegraphics[width=.3\textwidth]{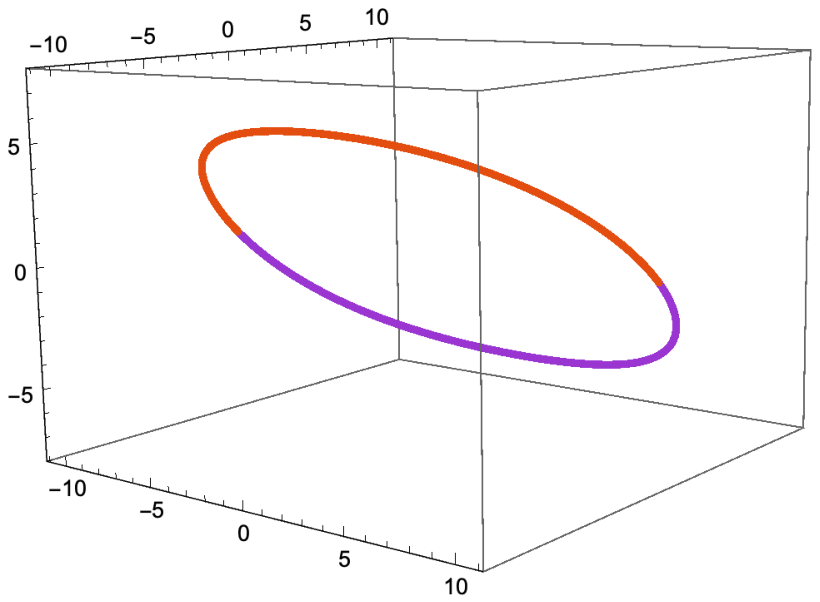}
\end{center}
\caption{This figure shows the graph of  the function $y^1(t)$ for the periodic solutions provided by the points $p_2$, $p_3$ and $p_4$ respectively. }\label{threeorbits}
\end{figure}

For the periodic solution given by the solution $p_4$, Figure \ref{fig2} shows the image of the functions $x(t)$, $y^1(t)$, $y^2(t)$ and $y^3(t)$.

\begin{figure}[hbtp]
\begin{center}\includegraphics[width=.45\textwidth]{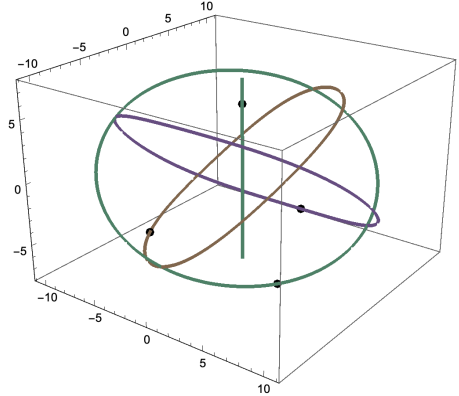}
\end{center}
\caption{For this periodic solution $m=3$, $M=7$, $r_0=11$ and $n=3$. This solutions was found by first finding the bifurcation point given by Theorem \ref{red1} and then finding a solution with $b\ne 0$ that is nearby, and doing analytic continuation of this latter point. For this solution $\theta(T)=\pi$. }\label{fig2}
\end{figure}

 \end{document}